\documentclass[11pt]{article}
\usepackage{latexsym,amsmath,amsthm,amssymb,graphicx,mathrsfs,amsfonts,yfonts,calrsfs,eufrak,bm,stackrel}
\input epsf
\usepackage{tikz}
\usepackage{relsize}
\usepackage{mathtools}
\DeclareFontEncoding{LS1}{}{}
\DeclareFontSubstitution{LS1}{stix}{m}{n}
\DeclareSymbolFont{symbols3}{LS1}{stixbb}{m}{n}
\DeclareMathSymbol{\bigslopedvee}{\mathbin}{symbols3}{"A7}

        {\qquad\hspace*{\fill}$\square$\end{trivlist}}

        {\qquad\hspace*{\fill}\end{trivlist}}

\newcounter{fig}

\newcounter{clai}
\newcounter{exa}
\newcounter{theorem}[section]
\renewcommand{\thetheorem}{\arabic{section}.\arabic{theorem}}
\newcounter{nona}[theorem]
\newcounter{nonanona}[theorem]

\renewcommand{\thenona}{\Alph{nona}}
\renewcommand{\thenonanona}{\alph{nonanona}}

\newenvironment{theorem}{\begin{trivlist}\item[]\refstepcounter{theorem}%
        {\bf\thetheorem\ Theorem}\par\nobreak\noindent\sl\ignorespaces}{%
        \ifvmode\smallskip\fi\end{trivlist}}

\newenvironment{theoremplus}[1]{\begin{trivlist}\item[]%
        \refstepcounter{theorem}{\bf\thetheorem\ Theorem} {\rm (\,#1\,)}%
        \par\nobreak\noindent\sl\ignorespaces}{%
        \ifvmode\smallskip\fi\end{trivlist}}
\newenvironment{lemma}{\begin{trivlist}\item[]\refstepcounter{theorem}%
        {\bf\thetheorem\ Lemma}\par\nobreak\noindent\sl\ignorespaces}{%
        \ifvmode\smallskip\fi\end{trivlist}}

\newenvironment{conjectureplus}[1]{\begin{trivlist}\item[]%
        \refstepcounter{theorem}{\bf\thetheorem\ Conjecture} %
        {\rm(\,#1\,)}\par\nobreak\noindent\sl\ignorespaces}{%
        \ifvmode\smallskip\fi\end{trivlist}}

\newenvironment{observation}{\begin{trivlist}\item[]\refstepcounter{theorem}%
        {\bf\thetheorem\ Observation}\par\nobreak\noindent\ignorespaces}{%
        \ifvmode\smallskip\fi\end{trivlist}}
\newenvironment{question}{\begin{trivlist}\item[]%
        \refstepcounter{theorem}{\bf\thetheorem\ Question}%
        \par\nobreak\noindent\sl\ignorespaces}{%
        \ifvmode\smallskip\fi\end{trivlist}}
\newenvironment{problem}{\begin{trivlist}\item[]%
        \refstepcounter{theorem}{\bf\thetheorem\ Problem}%
        \par\nobreak\noindent\sl\ignorespaces}{%
        \ifvmode\smallskip\fi\end{trivlist}}

        {\endlist\unskip}

        {\endtrivlist}
\newcommand{\wideitem}[1]{\leavevmode\hangindent\mathindent\noindent%
        \hbox to \mathindent{#1\hfil}\ignorespaces}

\newcommand{\eop}{\rule{1.4ex}{1.4ex}}

\newcommand{\B}{\mathcal{B}}

\newcommand{\U}{{\cal U}}

\newcommand{\M}{{\cal M}}

\newcommand{\V}{{\cal V}}

\title{\bf Serial exchanges in random bases}

\author{Sean McGuinness \\ Dept. of Mathematics\\ Thompson Rivers University\\
McGill Road, Kamloops BC\\ V2C5N3 Canada\\ email: smcguinness@tru.ca}

\date{}

\begin{document}

\maketitle

\begin{abstract}
It was conjectured by Kotlar and Ziv \cite{KotZiv} that for any two bases $B_1$ and $B_2$ in a matroid $M$ and any subset $X \subset B_1$, there is a subset $Y$ and orderings $x_1 \prec x_2 \prec \cdots \prec x_k$ and $y_1 \prec y_2 \prec \cdots \prec y_k$
of $X$ and $Y$, respectively, such that for $i = 1, \dots ,k$, $B_1 - \{ x_1, \dots ,x_i\} + \{y_1, \dots ,y_k \}$ and $B_2 - \{ y_1, \dots ,y_i\} + \{x_1, \dots ,x_k \}$ are bases; that is, $X$ is {\it serially exchangeable} with $Y$.    Let $M$ be a rank-$n$ matroid which is representable over $\mathbb{F}_q.$ We show that for $q>2,$ if bases $B_1$ and $B_2$ are chosen randomly amongst all bases of $M$, and if a subset $X$ of size $k \le \ln(n)$ is chosen randomly in $B_1$, then with probability tending to one as $n \rightarrow \infty$, there exists a subset $Y\subset B_2$ such that $X$ is serially exchangeable with $Y.$

\end{abstract}

\section{Introduction}

Let $\B(M)$ denote the set of bases in a matroid $M.$  For convenience, if we create a new basis from a basis $B$ by deleting $X\subset B$ and adding elements of $Y,$ then we denote the resulting basis by $B-X+Y$.  In the case where $X = \{ x_1, \dots ,x_k \}$ and $Y= \{ y_1, \dots ,y_k \},$ we will often write the new basis as $B - x_1 - \cdots - x_k +y_1 + \cdots + y_k.$  For example, if $X = \{ e \}$ and $Y = \{ f \},$ then the new basis is just $B-e+f.$

For all positive integers $k$, we let $[k]$ denote the set $\{ 1, \dots ,k \}.$  For positive integers $a, b$ where $a \le b,$ we let $[b\backslash a]$ denote the set $[b] - [a].$ 

Let $B_i \in \B(M),\ i = 1,2$.  For elements $x\in B_1$ and $y\in B_2$, we say that $x$ and $y$ are {\bf symmetrically exchangeable}  with respect to $B_i,\ i = 1,2$ if
$B_1'=B_1 -x +y$ and $B_2' = B_2-y+x$ are bases.  

We have the following well-known {\bf symmetric exchange property}:  for any element $x\in B_1$ there exists $y\in B_2$ for which $x$ and $y$ are symmetrically exchangeable.  For bases $B_i,\ i = 1,2$ and subsets $X_i \subseteq B_i,\ i = 1,2$ we write
$B_1, X_1 \rightarrow B_2, X_2$ if $B_2 - X_2 + X_1$ is a basis (and we write $B_1, X_1 \leftarrow B_2, X_2$ if $B_1 - X_1 + X_2$ is a basis).  We write $B_1, X_1 \leftrightarrow B_2, X_2$ if we have both $B_1, X_1 \rightarrow B_2, X_2$ and $B_1, X_1 \leftarrow B_2, X_2$.
It was shown by Greene \cite{Gre} and Woodall \cite{Woo} that the symmetric exchange property can be generalized to sets:   

\begin{theoremplus}{Greene, Woodall}
For every non-empty subset $X_1 \subseteq B_1$ there exists a subset $X_2 \subseteq B_2$ such that $B_1, X_1 \leftrightarrow B_2, X_2$. \label{the-Woodall}
\end{theoremplus}

For $i = 1,2,$ let $B_i$ be a basis in a matroid $M$ and let $X_i$ be a subset of $B_i.$  
%
If for some ordering $x_{11} \prec x_{12} \prec \cdots \prec x_{1k}$ of $X_1$ and some ordering $x_{21} \prec x_{22} \prec \cdots \prec x_{2k}$ of $X_2$, $B_1 - \{ x_{11}, \dots ,x_{1i} \} + \{x_{21}, \dots ,x_{2i} \}$ and $B_2 - \{ x_{21}, \dots ,x_{2i} \} + \{x_{11}, \dots ,x_{1i} \}$ are bases for $i = 1, \dots ,n$, then we write  $B_1,X_1 \stackrel[M]{}{\leftrightarrows} B_2,X_2$ (dropping $M$ when it is implicit).  In this case, 
we say that $X_1$ is {\bf serially exchangeable} with $X_2$ with respect to $B_i,\ i = 1,2.$  We refer to such orderings of $X_i,\ i=1,2$ as {\bf serial orderings}.  For a matroid $M$ and positive integer $k$, we say that $M$ has the $\mathbf{k}${\bf- serial exchange property} if for any two bases $B_1,B_2 \in \B(M)$ and any subset $X_1 \subseteq B_1$ where $|X_1| =k,$ there is a subset $X_2 \subseteq B_2$ where $B_1,X_1 \stackrel[]{}{\leftrightarrows} B_2,X_2$.

By the symmetric exchange property, all matroids have the $1$-serial exchange property. In   \cite{KotZiv}, Kotlar and Ziv made the following interesting conjecture:

\begin{conjectureplus}{Kotlar, Ziv \cite{KotZiv}}
For all matroids $M$ and for all integers $k \le r(M)$, $M$ has the $k$-serial exchange property. \label{con-KotZiv}
\end{conjectureplus}

In \cite{KotZiv}, it was shown that all matroids have the $2$-serial exchange property.  
Furthermore, in \cite{Kot}, it was shown that for matroids of rank at least three, for any two bases $B_1,B_2$, there exist $3$-subsets $A_i\subseteq B_i,\ i = 1,2$ such that $A_1$ is serially exchangeable with $A_2.$ 
In \cite{Mcg}, it was shown that for the set of matroids $\M_q$ representable over $\mathbb{F}_q,$ it suffices to prove Conjecture \ref{con-KotZiv} for matroids $M \in \M_q$ having rank at most $(k+2)q^{2k}.$  In addition, it is proven that all binary matroids have the $3$-serial exchange property.  Conjecture \ref{con-KotZiv} implies the following well-known conjecture of Gabow \cite{Gab} and Cordovil, Moreira \cite{CorMor}:

\begin{conjectureplus}{Gabow, Cordovil and Moreira}
Let $B_i$, for $i = 1,2,$ be disjoint bases of a matroid $M$ having rank $n.$ Then there exists orderings $b_1 \prec b_2 \prec \cdots \prec b_n$ and $b_{n+1} \prec b_{n+2} \prec \cdots \prec b_{2n}$ of the elements of $B_1$ and $B_2$ respectively, such that any $n$ consecutive elements in the cyclic ordering $b_1 \prec b_2 \prec \cdots \prec b_n \prec b_{n+1} \prec \cdots \prec b_{2n}$ form a basis in $M.$\label{con-GCMZ}
\end{conjectureplus} 

In other words, the above conjecture asserts that for disjoint bases $B_1$ and $B_2$, the base $B_1$ is serially exchangeable with $B_2.$  This conjecture in known to be true for transversal matroids (see \cite{Far,Gab}) and strongly base orderable matroids.  In \cite{FarRicSha, Wie}, this conjecture was proven for graphic matroids.   In \cite{Bon}, it was verified for sparse paving matroids.  

Conjecture \ref{con-KotZiv} remains open even in the case $k=3$ and there appears to be little prospect for a significant advance.  However, there are some interesting questions that arise when bases are chosen randomly.  In a recent paper, Sauermann \cite{Sau} proved that Rota's basis conjecture (see \cite{HuaRot}) holds for almost all bases chosen in a matroid representable over a finite field.  That is, if $n$ bases are chosen randomly from a rank-$n$ matroid representable over a finite field, then one find $n$ independent transversals of these bases with probability $1 - o(1)$ as $n\rightarrow \infty.$  In light of the conjecture above, this result leads to the following natural question:

\begin{question}
Suppose one randomly chooses bases $B_1$ and $B_2$ from a matroid and then randomly chooses a $k$-subset $X\subseteq B_1.$  What is the probability that there is a $k$-subset $Y\subseteq B_2$ such that $X$ is serially exchangeable with $Y$?
\end{question}

In this paper, we address this question for matroids representable over a finite field, with the exception of binary matroids.  The next theorem is the main result of the paper.


\begin{theorem}
Let $B_1$ and $B_2$ be randomly chosen bases in $\mathbb{F}_q^n$ and let $X$ be a randomly chosen $k$-subset in $B_1$ where $k \le \ln(n).$  If $q>2,$ then the probability that there exists a $k$-subset $Y\subseteq B_2$ for which  $X$ is serially exchangeable with $Y$ tends to one as $n \rightarrow \infty.$\label{the-main} 
\end{theorem}


\section{Notation}

For an $m\times n$ matrix $A= [a_{ij}]$ and subsets $S \subseteq [m],$ $T \subseteq [n],$ we denote by $A_{S,T}$ the submatrix of $A$ induced by the elements $a_{ij}$ where $i\in S$ and $j \in T.$  

In the model we use, each base will have equal probability of being chosen.  Rather than just choosing the bases themselves, we shall select {\it ordered} bases at random.  Given that each base has the same number of orderings, it will suffice to prove Theorem \ref{the-main} for ordered bases.  Each ordered base $B$ of $M$ corresponds to an $n$-tuple $(\mathbf{b}_1, \dots ,\mathbf{b}_n) \in \mathbb{F}_q^n$ such that $\{\mathbf{b}_1, \dots ,\mathbf{b}_n\}$ has rank $n.$  

Let $B_1$ and $B_2$ be randomly chosen (ordered) bases.  Let $\mathbf{U}_1, \dots ,\mathbf{U}_n$ be random variables with values in $\mathbb{F}_q^n$  such that the $n$-tuple $(\mathbf{U}_1, \dots ,\mathbf{U}_n)$ of random variables corresponds to $B_1$. For $i = 1, \dots ,n$, let $\mathbf{U}_i = (u_{1i}, u_{2i}, \dots ,u_{ni})$ where the $u_{ji}$'s are also random variables having values in $\mathbb{F}_q.$  Similarly, let $\mathbf{V}_i,\ i = 1, \dots ,n$ be random variables having values in $\mathbb{F}_q^n$ where the $n$-tuple $(\mathbf{V}_1, \dots ,\mathbf{V}_n)$ corresponds to $B_2.$  For $i = 1, \dots ,n$, let $\mathbf{V}_i = (v_{1i}, v_{2i}, \dots ,v_{ni}).$

Let $k$ be a positive integer where we will assume $k \le \ln (n)$ and let $\ell = \lfloor \frac nk \rfloor.$  For $i = 1, \dots ,\ell$ let $\mathbf{\V}_{i} = \{ \mathbf{V}_{(i-1)k+1}, \dots ,\mathbf{V}_{ik} \}$; that is, $\mathbf{\V}_{i}$ corresponds to the $k$-subset  of elements in $B_2$ in positions $(i-1)k+1, \dots ,ik.$  

After picking $B_1$ and $B_2$, we shall choose a set of $k$ elements from $B_1$ at random.  By symmetry, we may assume that these elements are the first $k$ elements in the (ordered) basis $B_1.$  
We let $\U_1 = \{ \mathbf{U}_1, \dots ,\mathbf{U}_k \},$ which corresponds to the random set of $k$ elements chosen from $B_1.$  To prove Theorem \ref{the-main} it will suffice to show that the probability that for some $i$, $B_1,\U_1 \stackrel[M]{}{\leftrightarrows} B_2,\V_i$, tends to one as $n\rightarrow \infty.$  To do this, 
we define random variables $X_1, \dots , X_{\ell}$  and $Y_1, \dots ,Y_{\ell}$ which depend on the randomly selected $B_1$ and $B_2$ as follows: for $j=1, \dots , \ell$, 
 
 $$X_j = \left\{ \begin{array}{lr} 1 &\mathrm{if} \ B_1, \U_1  \rightarrow B_2, \V_{j}\\ 0 & \mathrm{otherwise} \end{array} \right. \ \ \mathrm{and} \ \  
Y_j  =  \left\{ \begin{array}{lr} 1 &\mathrm{if} \  B_1, \U_1  \leftarrow B_2, \V_{j}  \\ 0 & \mathrm{otherwise} \end{array} \right.$$  It should be noted that the random variables $X_1, \dots , X_n$ are not necessarily independent and the same applies to $Y_1, \dots ,Y_n.$

\section{Lower bounds for probabilities}
In this section, we calculate lower bounds for the probabilities\\ $\mathbb{P}(X_i=1 \big| \ X_1, \dots X_{i-1})$ and $\mathbb{P}(Y_i=1 \big| \ Y_1, \dots ,Y_{i-1})$.
The following lemma is a well-known fact and we include its proof for completeness.
\begin{lemma}
The probability that a randomly chosen $k\times k$ matrix over $\mathbb{F}_q$ is non-singular is
$$\alpha_k = \prod_{i=1}^k (1 - \frac 1{q^i}).$$\label{lem-nonsingular}
\end{lemma} 

\begin{proof}
We first calculate the number of $k\times k$ nonsingular matrices $A$ there are over $\mathbb{F}_q.$  To do this, we fill in the rows of $A$ one at-a-time.  There are $q^k -1$ ways of filling in row $1.$
Filling in the second row, there are $q^k$ possible selections with the caveat that no selection can be a multiple of the first row.  Thus there are $q^k -q$ ways to fill in the second row.  Suppose now that we have filled in rows $1, \dots ,i.$  Again, there are $q^k$ choices for row $i+1$ with the caveat that no such selection can be a linear combination of rows $1, \dots ,i.$  Thus there are $q^k - q^i$ selections for row $i+1$.  Continuing, we see that in all there are $\prod_{i=1}^k (q^k - q^{i-1})$ non-singular matrices.  Since altogether there are
$q^{k^2}$ possible matrices $A$, the probability that $A$ is singular is $$\frac {\prod_{i=1}^k (q^k - q^{i-1})}{q^{k^2}} = \frac {q^{\frac{k^2-k}2}\prod_{i=1}^k (q^i - 1)}{q^{k^2}} = \prod_{i=1}^k \left( \frac {q^i -1}{q^i} \right) =  \prod_{i=1}^k (1 - \frac 1{q^i}).$$
\end{proof}

From {\it Euler's Theorem} (see \cite[Theorem 353]{HarWri}), we have $\prod_{i=1}^\infty (1-x^n) = 1 - x -x^2 + x^5 + x^7 - x^{12} - x^{15} \cdots$.  In particular, this theorem implies that for $q >1,$ $\alpha = \prod_{i=1}^\infty (1 - \frac 1{q^i}) > 1 - \frac 1q - \frac 1{q^2}.$

We say that a $k\times k$ matrix $A$ has {\bf sequential full rank} if for $i = 1, \dots ,k$, the submatrix $A_{[i],[i]}$ has full rank.
 
 \begin{observation}
 There are $(q-1)^k q^{k^2 -k}$ matrices $A$ over $\mathbb{F}_q$ for which $A$ has sequential full rank.\label{obs2}
 Moreover, the probability that a randomly selected $k\times k$ matrix has sequential full rank equals $(1 - \frac 1q)^k.$
 \end{observation}
 
\begin{proof}
Let $A = (a_{ij})$ be any $k\times k$ matrix over $\mathbb{F}_q$.  We shall fill in the entries for $A_{[i],[i]}$ in the order $i = 1,2, \dots, k$ so as to obtain a matrix with sequential full rank.  The matrix $A_{[1],[1]}$, which has only one entry, can be any non-zero element of $\mathbb{F}_q.$  Thus there are $q-1$ possible entries.
Next, we fill in the remaining $3$ entries of $A_{[2],[2]}$, one-by-one. The entries $a_{12}$ and $a_{21}$ can be any element of $\mathbb{F}_q,$ but to assure that $det(A_{[2],[2]}) \ne 0,$ there are only $q-1$ possible elements for the remaining entry $a_{22}.$  Thus there are $(q-1)^2q^2$ possible matrices $A_{[2],[2]}.$  Continuing inductively, suppose there are $(q-1)^{j}q^{j^2 - j}$ possible matrices for $A_{[j],[j]}$ having sequential full rank.  We now fill in the remaining $2(j+1)-1 = 2j+1$ entries of $A_{[j+1],[j+1]}$ one-by-one.  The $2j$ entries other than $a_{(j+1)(j+1)},$ can be any element in $\mathbb{F}_q$. Assuring that $det(A_{[j+1],j[+1]}) \ne 0,$ means that there are only $q-1$ possible values for the last entry $a_{(j+1)(j+1)}.$  Note that here we use the assumption that $det(A_{[j],[j]}) \ne 0.$  Thus there are\\ $(q-1)^{j}q^{j^2 - j} \cdot (q-1)q^{2j} = (q-1)^{j+1} q^{(j+1)^2 - (j+1)}$ selections for $A_{[j+1],[j+1]}.$  It follows by induction that there
are $(q-1)^k q^{k^2 -k}$ $k\times k$ matrices $A$ having sequential full rank.  Given that there are $q^{k^2}$ possible $k\times k$ matrices, the probability a randomly chosen $k\times k$ matrix has sequential full rank equals\\ $\frac {(q-1)^k q^{k^2 -k}}{q^{k^2}} = (1 - \frac 1q)^k.$     
\end{proof}

Let $\beta_k = (1 - \frac 1q )^k.$
Let $\overline{U} = (u_{ij})$ be the $n\times n$ (random) matrix whose $i$'th column corresponds to $\mathbf{U}_i$, and let $\overline{V} = (v_{ij})$ be the $n\times n$ (random) matrix whose $i$'th column corresponds to $\mathbf{V}_i.$  Such matrices $\overline{U}$ and $\overline{V}$ of full rank correspond in a one-to-one fashion with the ordered bases $B_1$ and $B_2.$

For $i = 0, \dots ,\ell$, let  $n_i = i k$ and let $J_i = [n_i\backslash n_{i-1}] = \{ n_{i-1}+1, \dots ,n_i \}.$

\begin{lemma}
For $i = 1, \dots ,\ell,$ $\mathbb{P}(X_i =1 \big| \ X_1, \dots ,X_{i-1}) \ge \alpha_k.$ \label{lem2}
\end{lemma}
\begin{proof}
It suffices to prove the lemma when $B_2$ is given.  Since there is an isomorphism mapping any base into another, we may assume that $B_2$ is the standard basis for $\mathbb{F}_q^n;$  that is, we may assume that for $i=1, \dots ,n,$ $\mathbf{V}_i = \mathbf{e}_i.$
 
 We shall choose $B_1$ by choosing the corresponding rank-$n$ matrix $\overline{U}$ randomly amongst all $n\times n$ matrices of full rank.  We do so by first choosing the first $k$ columns of $\overline{U}$ (that is, $\overline{U}_{[n],[k]}$) randomly among all $n\times k$ rank-$k$ matrices.   We select the entries for $\overline{U}_{[n],[k]}$ by selecting the entries for $\overline{U}_{J_i,[k]},$ in the order $i=1,2, \dots ,\ell$ and then selecting the remaining entries for the rest of $\overline{U}_{[n],[k]}$.    
 
 Since $B_2$ is the standard basis, we have $B_1, \U_1 \rightarrow B_2, \V_{i}$ if and only if $\overline{U}_{J_i,[k]}$ has full rank. Thus $X_i =1$ if and only if $\overline{U}_{J_i,[k]}$ has full rank.
Consequently, for each $i$, the probability $\mathbb{P}(X_i=1 \big| \ X_1, \dots X_{i-1}),$ will be at least the probability that $\overline{U}_{J_i,[k]}$ has full rank.  By Lemma \ref{lem-nonsingular}, this probability equals $\alpha_k = \prod_{i=1}^k (1 - \frac 1{q^i}).$ Thus $\mathbb{P}(X_i =1\big| \ X_1, \dots ,X_{i-1}) \ge \alpha_k.$
  
\end{proof}

\begin{lemma}
For $i = 1, \dots ,\ell,$ $\mathbb{P}(Y_i = 1 \big| \ Y_1, \dots ,Y_{i-1}) \ge \alpha_k.$  \label{lem1}
\end{lemma}

\begin{proof}
It suffices to prove the lemma when $B_1$ is given.  Since for any two bases in $\mathbb{F}_q^n$ there is an isomorphism mapping one to the other, we may assume that $B_1$ is the standard basis; that is,
$\mathbf{U}_i = \mathbf{e}_i,\ i = 1, \dots ,n.$
We now construct our basis $B_2$ by choosing the matrix $\overline{V}$ randomly among all $n\times n$ matrices of full rank.  We do that by choosing $\overline{V}_{[n], J_1}, \overline{V}_{[n], J_2}, \dots ,\overline{V}_{[n], J_\ell}$ in this order, and then selecting the remaining entries for $\overline{V}.$
Assume that we are given $\overline{V}_{[n], [n_{i-1}]};$ that is, $\overline{V}_{[n], J_1}, \dots ,\overline{V}_{[n],J_{i-1}}$ have be selected where $rank (\overline{V}_{[n],[n_{i-1}]}) = n_{i-1}.$  
Among all $n\times k$ matrices of rank $k,$ we choose $\overline{V}_{[n], J_i}$ as follows:   
First, we randomly choose elements for $\overline{V}_{[k],J_i},$ with the restriction that it have full rank.  This will insure that $B_1, \U_1 \leftarrow B_2, \V_i$.
Next, we fill in the remaining columns of $\overline{V}_{[n],J_i}$ with the restriction that the matrix $\overline{V}_{[n],[n_i]}$ have rank $n_i.$
As was seen in the proof of Lemma \ref{lem-nonsingular}, there are 
$q^{\frac {k^2-k}2}\prod_{j=1}^k(q^i -1)$ possible $k\times k$ matrices having full rank.  Thus we have as many choices for $\overline{V}_{[k],J_i}$.   
Let $r_0 = rank(\overline{V}_{[k],[n_{i-1}]})$ and for $j = 1, \dots ,k$, let  $r_j = rank(\overline{V}_{[k], [n_{i-1} +j]}),$ noting that $r_k = k.$ 

 For $p=1, \dots ,k$ let $\mathbf{w}_p$ be the $p$'th column vector in $\overline{V}_{[n],J_i}$.  Furthermore, let $\mathbf{w}_{[k],p} \in \mathbb{F}_q^k$ be the $p$'th column vector of $\overline{V}_{[k],J_i}$ and let $\mathbf{w}_{[n\backslash k],p}$ be the $p$'th column vector of  $\overline{V}_{[n\backslash k],J_i}.$   That is, $\mathbf{w}_{[k],p}$ is the vector $\mathbf{w}_p$ restricted to its first $k$ components, and $\mathbf{w}_{[n\backslash k],p}$ is the vector $\mathbf{w}_p$ restricted to its last $n-k$ components.
 We shall assume that $\mathbf{w}_{[k],p},\ p=1, \dots ,k$ have been selected so that $rank (\overline{V}_{[k],[J_i]}) =k$.  It remains to select $\mathbf{w}_{[n\backslash k],p},\ p=1, \dots ,k$ in such a way that $rank (\overline{V}_{[n], [n_i]}) = n_i.$  
 
 Suppose $\mathbf{w}_{[k],1}$ is in the column space of $\overline{V}_{[k],[n_{i-1}]};$ that is, \\ $r_1 = rank(\overline{V}_{[k], [n_{i-1} +1]}) =r_0.$  Given that $\overline{V}_{[k],[n_{i-1}]}$ has rank $r_0$, there are exactly $q^{n_{i-1} -r_0}$ linear combinations of the columns of $\overline{V}_{[k],[n_{i-1}]}$ which equal $\mathbf{w}_{[k],1}$.  Thus there are  $q^{n-k} - q^{n_{i-1} -r_0}$ selections for $\mathbf{w}_{[n\backslash k],1}$ for which the matrix $\overline{V}_{[n],[n_{i-1}+1]}$ has rank $n_{i-1}+1.$  
 
 Suppose instead that $\mathbf{w}_{[k],1}$ does not belong to the column space of $\overline{V}_{[k], [n_{i-1}]}.$  Then choosing any vector $\mathbf{w}_{[n\backslash k],1} \in \mathbb{F}_q^{n-k}$ will result in the matrix $\overline{V}_{[n],[n_{i-1}+1]}$ having rank $n_{i-1}+1.$ Thus there are $q^{n-k}$ selections for $\mathbf{w}_{[n\backslash k],1}$ in this case.  
 
 In general, let $p \in [k]$ and assume that the vectors $\mathbf{w}_{[n\backslash k],1}, \dots ,\mathbf{w}_{[n\backslash k],p-1}$ have been selected so that $rank (\overline{V}_{[n], [n_{i-1} + p-1]}) = n_{i-1} + p-1.$
   We will now determine the number of selections for $\mathbf{w}_{[n\backslash k],p}.$  If $r_p = r_{p-1}  = r,$ then there are exactly $q^{n_{i-1}+p-1 -r}$ linear combinations of the columns of $\overline{V}_{[k],[n_{i-1}+p-1]}$ which equal $\mathbf{w}_{[k],p}$.
 Thus there are $q^{n-k} - q^{n_{i-1}+p-1 -r}$ selections for $\mathbf{w}_{[n\backslash k],p}$  for which $\overline{V}_{[n],[n_{i-1}+p]}$ has rank $n_{i-1}+p.$
 If $r_p =  r_{p-1}+1,$ then $\mathbf{w}_{[k],p}$ is not in the column space of $\overline{V}_{[k], [n_{i-1}+p-1]}$ and hence one can choose any vector $\mathbf{w}_{[n\backslash k],p} \in \mathbb{F}_q^{n-k }$ and $\overline{V}_{[n], [n_{i-1}+p]}$ will have rank $n_{i-1} +p.$  Thus there are $q^{n-k}$ selections for $\mathbf{w}_{[n\backslash k],p}$ in this case. 
 
 Given that $rank(\overline{V}_{[k], J_i}) = k,$ there are exactly $k_0 = k-r_0$ integers $p\in [k]$ for which $r_p = rank(\overline{V}_{[k],[n_{i-1}+p]}) > r_{p-1} = rank(\overline{V}_{[k], [n_{i-1} + p-1]})$ and we let $p_1 < \cdots < p_{k_0}$ be such integers.  Let $p_0 = 0$. 
For $j=1, \dots ,k_0,$ let $\beta_j = p_j - p_{j-1} -1$.  Let $\beta_{k_0 +1} = k - p_{k_0}.$  For $j = 0, \dots ,k_0,$ let $s_j = r_{p_j}.$
 We see that $s_j = r_0 + j,\ j = 1, \dots ,k_0.$  Also, $k_0 + \beta_1 + \beta_2 + \dots  + \beta_{k_0+1} =k$ and hence $\beta_1 + \dots \beta_j + j-1 + (k - s_{j-1}) + \beta_{j+1} + \cdots + \beta_{k_0+1} = k$. Observing that $\beta_1 + \cdots + \beta_{j-1} + j-1 = p_{j-1},$ it follows from the previous equation that
 
 $$k- s_{j-1} + p_{j-1} =  k - \beta_j - \cdots - \beta_{k_0+1}.$$  
 and hence $$s_{j-1} - p_{j-1} =  \beta_j + \cdots + \beta_{k_0+1}.$$
 
 Suppose $\beta_j = p_j - p_{j-1} -1 >0$ for some $j \in [k_0].$ Then each of the vectors $\mathbf{w}_{[k],p_{j-1} +1},\cdots , \mathbf{w}_{[k],p_{j-1} +\beta_i}$ belong to the column space of $\overline{V}_{[k], n_{i-1} + p_{j-1}}$.  Thus the number of choices for the vectors
 $\mathbf{w}_{[n\backslash k],p_{j-1} +1},\cdots , \mathbf{w}_{[n\backslash k],p_{j-1} +\beta_i}$ is 
 \begin{align*}
 \prod_{u=1}^{\beta_j} (q^{n-k} - q^{n_{i-1} + p_{j-1} +u -1 - s_{j-1} }) &=  \prod_{u=1}^{\beta_j} (q^{n-k} - q^{(n_{i-1}-k) + k+ p_{j-1} +u -1 - s_{j-1} })\\ &= \prod_{u=1}^{\beta_j} (q^{n-k} - q^{(n_{i-1} -k) +  k - \beta_j - \cdots - \beta_{k_0+1} + (u-1)}).
 \end{align*}
 
 Let $N$ be the number of ways one can choose
 $\overline{V}_{[n],J_i}$ so that $\overline{V}_{[k],J_i}$ has full rank and the matrix  $\overline{V}_{[n], [n_i]}$ has rank $n_i$.  By the above we have,
 
\begin{align*}
N &\ge \prod_{i=0}^{k-1} (q^k -q^i) \prod_{p \in [k] - \{ p_1, \dots ,p_{k_0} \} } (q^{n-k} - q^{n_{i-1}+ p -1 - r_{p-1} })  \prod_{p\in \{ p_1, \dots ,p_{k_0} \} } q^{n-k}\\
&=  \prod_{i=0}^{k-1} (q^k -q^i)  \left(\prod_{\beta_j >0} \prod_{u=1}^{\beta_j} (q^{n-k} - q^{n_{i-1} + p_{j-1} +u -1 - s_{j-1} })\right)q^{k_0(n-k)}\\
&= \prod_{i=0}^{k-1} (q^k -q^i) \left(\prod_{\beta_j >0}  \prod_{u=1}^{\beta_j} \left(q^{n-k} - q^{(n_{i-1} -k) +  k - \beta_j - \cdots - \beta_{k_0+1} + (u-1)}\right) \right)q^{k_0(n-k)}
\end{align*}

Given that $k - \beta_1 - \cdots - \beta_{k_0+1} = k_0,$ one sees that $$\prod_{\beta_j >0}  \prod_{u=1}^{\beta_j} (q^{n-k} - q^{(n_{i-1} -k) +  k - \beta_j - \cdots - \beta_{k_0+1} + (u-1)}) = \prod_{j=k_0}^{k-1} (q^{n-k} - q^{(n_{i-1} -k) +j}).$$
Thus $$N \ge \prod_{i=0}^{k-1}(q^k - q^i) \cdot \prod_{j=k_0}^{k-1}(q^{n-k} - q^{n_{i-1}-k+j}) \cdot q^{k_0(n-k)}.$$

On the other hand, given $\overline{V}_{[n], n_{i-1}},$ the total number of ways in which $\overline{V}_{[n],J_i}$ can be chosen so that $rank(\overline{V}_{[n], n_i}) =n_i$ is $\prod_{j=0}^{k-1} (q^n - q^{n_{i-1} +j}).$
 Thus we have 
 \begin{align*}
 \mathbb{P} (Y_i =1\big| \ Y_1, \dots ,Y_{i-1}) &\ge \frac N{\prod_{j=0}^{k-1}(q^n - q^{n_{i-1} +j})} = \frac N{q^{k^2} \prod_{j=0}^{k-1}(q^{n-k} - q^{n_{i-1}-k +j})}\\
 &\ge \frac {\prod_{i=0}^{k-1}(q^k - q^i) \cdot \prod_{j=k_0}^{k-1}(q^{n-k} - q^{n_{i-1}-k+j}) \cdot q^{k_0(n-k)}}{q^{k^2} \prod_{j=0}^{k-1}(q^{n-k} - q^{n_{i-1}-k +j})}\\
 &= \frac  {q^{k_0(n-k)}\prod_{i=0}^{k-1} (q^k -q^i)}{q^{k^2} \prod_{j=0}^{k_0-1}(q^{n-k} - q^{n_{i-1}-k +j})}\\
 &\ge \prod_{i=0}^{k-1}\left( \frac {q^k - q^i}{q^k} \right) = \prod_{i=1}^{k} \left( 1 - \frac 1{q^i} \right) = \alpha_k. 
 \end{align*}
 \end{proof}

\subsection{The Chernoff bound}

Let $Z_1, \dots ,Z_n$ be $n$ independent $0,1$ random variables where for all $i,$\\ $\mathbb{P}(Z_i =1) = p.$  Then the sum $Z = \sum_{i=1}^n Z_i$ has a binomial $B(n,p)$ distribution
and we have the following well-known concentration inequality due to Chernoff (see \cite{AloSpe}):

\begin{theoremplus}{Chernoff Bound}
For all $0 \le t \le np,$ we have
$$\mathbb{P}(|Z - np| > t)  < 2 e^{-\frac {t^2}{3np}}.$$\label{the-Chernoff}
\end{theoremplus} 

As a direct consequence of this bound, we have
\begin{lemma}
For $Z \sim Bin(n,p),$ $\mathbb{P}(Z > (1 - \epsilon)np) > 1 - 2 e^{-\frac {\epsilon^2 np}3}.$\label{lem-Chernoff}
\end{lemma}

Let $X = \sum_{i=1}^\ell X_i$ and $Y = \sum_{i=1}^\ell Y_i$  By Lemmas \ref{lem2} and  \ref{lem1}, it follows that the probability distributions for $X$ and $Y$ dominate the probability distribution for a random variable $Z \sim Bin(\ell, \alpha_k).$  That is, for all $a>0$,\\
$\mathbb{P}(X \ge a) \ge \mathbb{P}(Z \ge a)$ and  $\mathbb{P}(Y \ge a) \ge \mathbb{P}(Z \ge a)$. Thus by Lemma \ref{lem-Chernoff} we have the following:

\begin{lemma}
$\mathbb{P}(X \ge (1-\epsilon)\ell\alpha_k) > 1 - 2 e^{-\frac {\epsilon^2 \ell \alpha_k}3}$ and $\mathbb{P}(Y \ge (1-\epsilon)\ell\alpha_k) > 1 - 2 e^{-\frac {\epsilon^2 \ell \alpha_k}3}$ \label{lem3}.
\end{lemma}

By fixing $\epsilon$ in the above lemma, and given $k \le \ln (n)$, we have $\ell = \lfloor \frac nk \rfloor \rightarrow \infty$ as $n\rightarrow \infty.$ Thus $\mathbb{P}(X \ge (1- \epsilon)\ell\alpha_k) \rightarrow 1$ and $\mathbb{P}(Y \ge (1- \epsilon)\ell\alpha_k) \rightarrow 1$ as $n \rightarrow \infty.$

 We define random variables $Z_i, \ Z'_i,\ i \in [\ell]$ where
$$Z_i = \left\{ \begin{array}{lr} 1 & \mathrm{if} \  \ B_1, \U_1  \leftrightarrow B_2, \V_{i}\\ 0 &\mathrm{otherwise} \end{array} \right. \ \ \ 
Z_i' = \left\{ \begin{array}{lr} 1 & \mathrm{if} \  \ B_1, \U_1  \stackrel[]{}{\leftrightarrows}B_2, \V_{i}\\ 0 &\mathrm{otherwise} \end{array} \right. $$

Let $Z = Z_1 + \cdots + Z_{\ell}$ and $Z' = Z'_1 + \cdots + Z'_{\ell}.$ Recall that $\alpha = \prod_{i=1}^\infty (1 - \frac 1{q^i}).$ 

\begin{observation}
For $\epsilon >0,$ $\mathbb{P}(Z \ge (2(1-\epsilon)\alpha -1) \ell ) > 1 - 4 e^{-\frac {\epsilon^2 \ell \alpha}3}.$ 
Furthermore, if $q >2,$ then there is a constant $c>0$ such that $\mathbb{P}(Z \ge c\ell) \rightarrow 1,$ as $n \rightarrow \infty.$
\label{lem-Zbound}
\end{observation}

\begin{proof}
By Lemma \ref{lem3}, we have for $\epsilon >0$, $\mathbb{P}(X \ge (1-\epsilon)\ell\alpha_k) > 1 - 2 e^{-\frac {\epsilon^2 \ell \alpha_k}3}.$ Since $\alpha = \prod_{i=1}^\infty (1 - \frac 1{q^i}) < \alpha_k$, it follows that
$\mathbb{P}(X \ge (1-\epsilon)\ell\alpha) > 1 - 2 e^{-\frac {\epsilon^2 \ell \alpha}3}.$  Likewise, we have
$\mathbb{P}(Y \ge (1-\epsilon)\ell\alpha) > 1 - 2 e^{-\frac {\epsilon^2 \ell \alpha}3}$.  Thus
\begin{align*}\mathbb{P}(X \ge (1 -\epsilon)\ell\alpha\ \mathrm{and}\ Y \ge (1 -\epsilon)\ell\alpha) &= 1 - \mathbb{P}(X < (1 -\epsilon)\ell\alpha\ \mathrm{or}\ Y < (1 -\epsilon)\ell\alpha)\\
&\ge 1 - \mathbb{P} (X < (1 -\epsilon)\ell\alpha) -  \mathbb{P} (X < (1 -\epsilon)\ell\alpha)\\
&\ge 1 - 4 e^{-\frac {\epsilon^2 \ell \alpha}3}
\end{align*}
Thus it follows that $\mathbb{P}(X+Y \ge 2(1-\epsilon)\ell\alpha) \ge 1 - 4 e^{-\frac {\epsilon^2 \ell \alpha}3}.$
Noting that $Z \ge X+Y - \ell,$ it follows that 
$\mathbb{P}(Z + \ell \ge (1-\epsilon)2\ell\alpha) \ge 1 - 4 e^{-\frac {\epsilon^2 \ell \alpha}3}$ and thus $\mathbb{P}(Z \ge (2(1-\epsilon)\alpha -1) \ell ) \ge 1 - 4 e^{-\frac {\epsilon^2 \ell \alpha}3}.$
By Euler's theorem, we have $\alpha > 1 - \frac 1q - \frac 1{q^2}.$  Thus when $q >2$,we have $\alpha > \frac 59$ and $\mathbb{P} (Z \ge (\frac 19 - \frac {10}9 \epsilon)\ell)  > 1 - 4 e^{-\frac {\epsilon^2 \ell \alpha}3}.$
In particular, choosing $\epsilon  <  \frac 1{10},$ there is a constant $c>0$ such that $\mathbb{P}(Z \ge c\ell) \rightarrow 1,$ as $n \rightarrow \infty.$
\end{proof}

We shall exploit the following observation.

\begin{observation}
Suppose $D_i,\ i = 1,2$ are bases in a matroid and $X_i \subseteq D_i,\ i = 1,2$ are subsets where $D_1, X_1 \rightarrow D_2, X_2.$  Then for any ordering $x_{21} \prec x_{21} \prec \cdots \prec x_{2k}$ of the elements of $X_2$, there is an ordering $x_{11} \prec x_{11} \prec \cdots \prec x_{1k}$ of the elements of $X_1$ such that for $i = 1, \dots ,k,$ $D_2 - \{ x_{21}, \dots ,x_{2i} \} + \{ x_{11}, \dots ,x_{1i} \}$ is a basis.
\label{obs0}
\end{observation}

\begin{proof}
Let $D_2' = D_2 - X_2 + X_1,$ which by assumption is a basis.  Assume that the elements of $X_2$ are ordered as $x_{21} \prec x_{22} \prec \cdots \prec x_{2k}$.    Since $|D_2 - x_{21}| < |D_2'|,$ there exists an element $x_{11} \in D_2'$ for which $D_2^1 =D_2 - x_{21} + x_{11}$ is a basis. We note that $x_{11} \in X_1.$ Next, since $|D_2^1 - x_{22}| < |D_2'|,$ there exists $x_{12} \in D_2'$ for which $D_2^2 = D_2^1 - x_{22} + x_{12}$ is a basis.  Again, we note that $x_{12} \in X_1$.  Now suppose that for some $1\le i <k$ we have that  $D_2^i = D_2 - \{ x_{21}, \dots ,x_{2i} \} + \{ x_{11}, \dots ,x_{1i} \}$ is a basis, where $\{ x_{11}, \dots , x_{1i} \} \subseteq X_1.$  Since $|D_2^i - x_{2(i+1)}| < |D_2'|,$ there exists $x_{1(i+1)} \in D_2'$ such that $D_2^{i+1} = D_2^i - x_{2j} + x_{1(j+1)}$ is a basis, and furthermore, we must have $x_{1(i+1)}\in X_1.$  Continuing, we generate the elements $x_{11}, x_{12}, \dots ,x_{1k}$ and this will be the desired ordering of the elements of $X_1.$ 
\end{proof}

Recall that $\beta_k = \left( 1 - \frac 1q \right)^k.$  We have the following lemma:

\begin{lemma}
Let $S \subseteq [\ell]$ where $s = |S|.$  Then $$\mathbb{P}(Z'_{i} =0, \ \forall i\in S \ \big| \ Z_{i} =1, \forall i \in S) \le \ (1 - \beta_k)^s.$$
In addition, $\mathbb{P}(Z' =0 \ \big| \ Z = s) \le (1 - \beta_k)^s.$
\label{lem-lem4}
\end{lemma}

\begin{proof}
%
Assume that $Z_{i} =1, \forall i \in S.$  Then $B_1, \U_1  \stackrel[]{}{\rightarrow }B_2, \V_{i}$ and $B_1, \U_1  \stackrel[]{}{\leftarrow }B_2, \V_{i}$, for all $i \in S.$   Let $i\in S.$  It follows from Observation \ref{obs0} that for some ordering $V_{i_1} \prec V_{i_2} \prec \cdots \prec V_{i_k}$ of $\V_i$ we have, for $j = 1, \dots ,k,$ $B_1 - \{ U_1, \dots ,U_j \} + \{ V_{i_1}, \dots ,V_{i_j} \}$ is a basis. 
In light of this, we may assume that for each $i\in S,$ we have already re-ordered the vectors in $\V_i$ so that for $j = 1, \dots ,k,$ $B_1 - \{ U_1, \dots ,U_j \} + \{ V_{n_{i-1}+1}, \dots ,V_{n_{i-1}+j} \}$ is a basis.  
Let $\overline{W}$ be the matrix obtained by concatenating $\overline{U}$ with $\overline{V}$ and let $\overline{W}'$ be the matrix obtained by performing row operations on $\overline{W}$ so that $\overline{V}$ is reduced to the identity matrix $I_n$ and $\overline{U}$ is reduced to a matrix $\overline{U}' = (u_{ij}').$ 
 Note that the the row operations do not change the independence of subsets of columns in the column space of $\overline{W}$; that is, a set of columns in $\overline{W}$ is linearly independent if and only if the corresponding subset of columns in $\overline{W}'$ is linearly independent.  Let $i\in S.$  Since $B_1, \U_1  \stackrel[M]{}{\rightarrow }B_2, \V_{i}$, it follows that $\overline{U}'_{J_i,[k]}$ has full rank.  Moreover, if $\overline{U}'_{J_i,[k]}$ has sequential full rank, then we see that for $j=1, \dots ,k,$ replacing the first $j$ columns of $I_n$ with first $j$ columns of $\overline{U}'_{[n],[k]}$ will yield a matrix of full rank.
Thus if $\overline{U}'_{J_i,[k]}$ has sequential full rank, then for $j=1, \dots ,k,$ $B_2 - \{ V_{n_{i-1} +1}, \dots ,V_{n_{i-1} +j} \} + \{ U_1, \dots ,U_j \}$ is a basis.  In this case, we see that $Z'_i =1.$  By Observation \ref{obs2}, the probability that 
$\overline{U}'_{J_i,[k]}$ has sequential full rank is $\beta_k = (1 - \frac 1q)^k$.  Thus $\mathbb{P}(Z_i' =1 \big| \ Z_i=1) \ge \beta_k$, and this bound is independent of the other random variables.  We now see that $\mathbb{P}(Z'_i = 0, \  \forall i\in S \ \big| \ Z_i = 1, \forall i\in S) \le (1-\beta_k)^s$.

The second assertion follows from the first statement.
\end{proof}

\subsection{Proof of Theorem \ref{the-main}}

Assume that $q >2.$  By Lemma \ref{lem-Zbound}, there is a constant $c >0$ such that $\mathbb{P}(Z \ge c\ell) \rightarrow 1,$ as $\ell \rightarrow \infty.$
%
%
To complete the proof of Theorem \ref{the-main}, it suffices to show that as $n\rightarrow \infty$, $\mathbb{P}(Z' > 0) \rightarrow 1.$  We have by Lemma \ref{lem-lem4} that
\begin{align*}
\mathbb{P}(Z' = 0) &= \sum_{s =0}^\ell \mathbb{P}(Z' = 0, Z=s) = \sum_{s =0}^\ell \mathbb{P}(Z' = 0\ \big| \  Z=s) \mathbb{P}(Z=s)\\ 
&\le \sum_{s =0}^\ell (1 - \beta_k)^s \mathbb{P}(Z=s) = \sum_{s < c\ell} (1 - \beta_k)^s \mathbb{P}(Z=s) +  \sum_{s \ge c\ell} (1 - \beta_k)^s \mathbb{P}(Z=s)\\
&\le \mathbb{P} (Z < c\ell) + (1 - \beta_k)^{c\ell} \mathbb{P}(Z \ge c\ell).
\end{align*}

We have that $\lim_{n\rightarrow \infty} \mathbb{P} (Z < c\ell) = 0.$   We claim that $\lim_{n\rightarrow \infty}(1- \beta_k)^{c\ell} = 0.$  It suffices to prove this when $k = \ln (n)$  (since $\beta_{k'} > \beta_k$ and $\frac n{k'} > \frac nk$ if $k' < k$).  Let $\delta = \ln (1 - \frac 1q ).$
Observing that $-1 < \delta <0$ we have
\begin{align*}
(1- \beta_k)^{c\ell} &= (1- (1-\frac 1q)^k )^{c\ell} = (1- e^{\delta k})^{c\ell}\\
&= (1 - e^{\delta \ln (n)})^{c\ell} = (1 - n^{\delta})^{c\ell} = \left((1- \frac 1{n^{-\delta}} )^{n^{-\delta}}\right)^{\frac {cn^{1+ \delta}}{\ln(n)}}\\ 
\end{align*}

Since $\lim_{n\rightarrow \infty} (1- \frac 1{n^{-\delta}} )^{n^{-\delta}} = e^{-1},$ it follows that $\lim_{n \rightarrow \infty}  \left((1- \frac 1{n^{-\delta}} )^{n^{-\delta}}\right)^{\frac {cn^{1+ \delta}}{\ln(n)}} =0.$  Thus $\lim_{n\rightarrow \infty} (1- \beta_k)^{c\ell} = 0.$
It now follows from the above that as $n\rightarrow \infty$, $\mathbb{P}(Z' > 0) \rightarrow 1.$  This completes the proof.

\section{Discussion}

Given that Theorem \ref{the-main} does not apply to binary matroids,  the following is a natural problem:

 \begin{problem}
 Let $M$ be a rank-$n$ binary matroid and let $k$ be a fixed integer.  Suppose one chooses randomly two bases $B_1$ and $B_2$ and then chooses randomly a $k$-subset $X \subseteq B_1.$  Is the probability of finding a $k$-subset $Y \subseteq B_2$ which is serially exchangeable with $X$ tending to one as $n\rightarrow \infty$ ?
 \end{problem}

\end{document}